\documentclass{amsart}
\usepackage{amsmath}

\newtheorem{theorem}{Theorem}
\theoremstyle{plain}

\newtheorem{lemma}{Lemma}

\numberwithin{equation}{section}
\input{tcilatex}

\begin{document}
\title{Difference sequence spaces derived by Generalized weighted mean}
\author{Harun Polat}
\address{Mu\c{s} Alparslan University Art and Science Faculty , Mathematics
Department , 49100 Mu\c{s} , Turkey .}
\email{h.polat@alparslan.edu.tr}
\author{Vatan Karakaya}
\address{Department of Mathematical Engineering , Yildiz Technical
University , Istanbul -TURKEY}
\email{vkkaya@yildiz.edu.tr}
\author{Necip \c{S}\.{I}M\c{S}EK}
\address{\.{I}stanbul Commerce University, Department of Mathematics, \"{U}sk%
\"{u}dar, Istanbul,TURKEY}
\email{necsimsek@yahoo.com}

\begin{abstract}
In this work, we define new sequence spaces by combining generalized
weighted mean and difference operator. Afterward, we investigate topological
structure which are completeness, $AK$-property, $AD$-property. Also, we
compute the $\alpha -$ , $\beta -$ and $\gamma -$ duals, and obtain bases
for these sequence spaces. Finally, necessary and sufficient conditions on
an infinite matrix belonging to the classes $\left( c\left( u,v,\Delta
\right) :\ell _{\infty }\right) $ and $\left( c\left( u,v,\Delta \right)
:c\right) $ are established.
\end{abstract}

\maketitle

\section{\protect\Large Introduction}

In studies on the sequence spaces, there are some basic approaches which are
determination of topologies, matrix mapping and inclusions of sequence
spaces $($see$;$ \cite{W.R}).These methods are applied to study the matrix
domain $\lambda _{A}$ of an infinite matrix $A$ defined by $\lambda _{A}$ $%
=\left\{ x=\left( x_{k}\right) \in w:Ax\in \lambda \right\} .$ Especially,
the weighted mean and the difference operators which are special cases for
the matrix $A$ have been studied extensively via the methods mentioned above.

In the literature, some new sequence spaces are defined by using the
generalized weighted mean and the difference operator or by combining both
of them. For example, in \cite{HKIZ1981}$,$ the difference sequence spaces
are first defined by K\i zmaz. Further, the authors including Ahmad and
Mursaleen \cite{ZUM1987}$,$ \c{C}olak and Et\cite{C.Et}, Ba\c{s}ar and Altay
Karakaya and Polat\cite{vatan polat}, and the others have defined and
studied new sequence spaces by considering matrices that represent
difference operators. The articles concerning this work can be found in the
list of references \cite{BAHP2006}, \cite{ALBA2007}, \cite{ETBASAR} and \cite%
{HPFB2007}. On the other hand, by using generalized weighted mean, several
authors defined some new sequence spaces and studied some properties of
these spaces. Some of them are as follows: Malkowsky and Sava\c{s} \cite%
{MS2004} have defined the sequence spaces $Z\left( u,v,\lambda \right) $
which consists of all sequences such that $G\left( u,v\right) -$transforms
of them are in $\lambda \in \left\{ \ell _{\infty },c,c_{0},\ell
_{p}\right\} .$ Ba\c{s}ar and Altay \cite{BAFB2006} have defined and studied
the sequence spaces of nonabsolute type drived by generalized weighted mean
over the paranormed spaces.

In this work, our purpose is to introduce new sequence spaces by combining
the generalized weighted mean and difference operator and also to
investigate topological structure{\tiny \ }which are completeness, $AK,AD$
properties, the $\alpha -,\beta -,\gamma -$duals, and the bases of these
sequence spaces$.$ In addition, we characterize some matrix mappings on
these spaces.

\section{\protect\Large Preliminaries and Notations}

By $w$, we denote the space of all real or complex valued sequences. Any
vector subspace of $w$ is called a sequence space. We write $\ell _{\infty }$%
, $c$ and $c_{0}$ for the spaces of all bounded, convergent and null
sequences, respectively. Also by $bs,cs,\ell _{1,}$ we denote the spaces of
all bounded, convergent and absolutely convergent series, respectively.

A sequence space $\lambda $ with a linear topology is called a $K$-space
provided each of the maps $p_{i}:\lambda \rightarrow C$ defined by $%
p_{i}\left( x\right) =x_{i}$ is continuous for all $i\in N$ ; where $C$
denotes the complex field and $N=\left\{ 0,1,2,...\right\} $. A $K$-space $%
\lambda $ is called an $FK$ space provided $\lambda $ is a complete linear
metric space. An $FK$-space whose topology is normable is called a $BK$%
-space. An $FK$-space $\lambda $ is said to have $AK$ property, if $\varphi
\subset \lambda $ and $\{e^{(k)}\}$ is a basis for $\lambda $, where $%
e^{(k)} $ is a sequence whose only non-zero term is a $1$ in $kth$ place for
each $k\in N$ and $\varphi =span\{e^{(k)}\}$, the set of all finitely
non-zero sequences. If $\varphi $ is dense in $\lambda $, then $\lambda $ is
called an $AD$-space, thus $AK$ implies $AD$. For example, the spaces $%
c_{0},cs$, and $\ell _{p}$ are $AK$-spaces, where $1<p<\infty .$

Let $\lambda ,\mu $ be any two sequence spaces and $A=(a_{nk})$ be an
infinite matrix of real numbers $a_{nk}$, where $n,k\in N$. Then, we write $%
Ax=((Ax)n)$, the $A-$transform of $x$, if $A_{n}\left( x\right)
=\sum\nolimits_{k}a_{nk}x_{k}$ converges for each $n\in N$. If $x\in \lambda 
$ implies that $Ax\in \mu $ then we say that $A$ defines a matrix mapping
from $\lambda $ into $\mu $ and denote it by $A:\lambda \rightarrow \mu .$
By $(\lambda :\mu )$, we mean the class of all infinite matrices $A$ such
that $A:\lambda \rightarrow \mu $. For simplicity in notation, here and in
what follows, the summation without limits runs from $0$ to $\infty $. We
write $e=(1,1,1,...)$ and $U$ for the set of all sequences $u=(u_{k})$ such
that $u_{k}=0$ for all $k\in N$. For $u\in U$, let $1/u=(1/u_{k})$. Let $%
u,v\in U$ and define the matrix $G(u,v)=(g_{nk})$ by 
\begin{equation*}
g_{nk}=\left\{ 
\begin{array}{c}
u_{n}v_{k}\text{ if }0\leq k\leq n \\ 
0\text{ \ \ if \ }k>n%
\end{array}
\right.
\end{equation*}
for all $k,n\in N,$ where $u_{n}$ depends only on $n$ and $v_{k}$ only on $k$%
. The matrix $G(u,v)$, defined above, is called as generalized weighted mean
or factorable matrix. The matrix domain $\lambda _{A}$ of an infinite matrix 
$A$ in a sequence space $\lambda $ is defined by $\lambda _{A}=\left\{
x=(x_{k})\in w:Ax\in \lambda \right\} $ which is a sequence space. Although
in the most cases the new sequence space $\lambda _{A}$ generated by the
limitation matrix $A$ from a sequence space $\lambda $ is the expansion or
the contraction of the original space $\lambda $.

The $continuous$ $dual$ $X^{^{\prime }}$ of a normed space $X$ is defined as
the space of all bounded linear functionals on $X$. If A is triangle, that
is $a_{nk}=0$ if $k>n$ and $a_{nn}\neq 0$ for all $n\in N$, and $\lambda $
is a sequence space, then $g\in \lambda _{A}^{^{\prime }}$ if and only if $%
f=g\circ A,\ f\in $ $\lambda ^{^{\prime }}$.

Let $X$ be a seminormed space. A set $Y\subset X$ is called fundamental if
the span of $Y$ is dense in $X$. One of the useful results on fundamental
set which is an application of Hahn--Banach Theorem as follows: If $Y$ is
the subset of a seminormed space $X$ and $f\in X^{^{\prime }}$, $f(Y)=0$
implies $f=0$, then $Y$ is fundamental [\cite{Wilansky} p. 39].

\section{{\protect\Large The sequence spaces }$\protect\lambda \left(
u,v,\Delta \right) ${\protect\Large \ for }$\protect\lambda \in \left\{ \ell
_{\infty },c,c_{0}\right\} $}

In this section, we define the new sequence spaces $\lambda \left(
u,v,\Delta \right) $ for $\lambda \in \left\{ \ell _{\infty
},c,c_{0}\right\} $ derived by the generalized weighted mean, and prove that
these are the complete normed linear spaces and compute their $\alpha
-,\beta -,$ and $\gamma -$ duals. Furthermore, we give the basis for the
spaces $\lambda \left( u,v,\Delta \right) $ for $\lambda \in \left\{
c,c_{0}\right\} .$ Also we show that these spaces have $AK$ and $AD$
properties.

We define the sequence spaces $\lambda \left( u,v,\Delta \right) $ for $%
\lambda \in \left\{ \ell _{\infty },c,c_{0}\right\} $ by 
\begin{equation*}
\lambda \left( u,v,\Delta \right) =\left\{ x=\left( x_{k}\right) \in w:\text{
}y=\sum\limits_{i=1}^{k}u_{k}v_{i}\Delta x_{i}\in \lambda \right\} .
\end{equation*}
We write $\Delta x=\left( \Delta x_{i}\right) $ for the sequence $\left(
x_{k}-x_{k-1}\right) $ and use the convention that any term with negative
subscript is equal to naught.

If $\lambda $ is any normed sequence space, then we call the matrix domain $%
\lambda _{G\left( u,v,\Delta \right) }$ as the generalized weighted mean
difference sequence space. It is natural that these spaces may also be
defined according to matrix domain as follows: 
\begin{equation*}
\lambda \left( u,v,\Delta \right) =\left\{ \lambda \right\} _{G\left(
u,v,\Delta \right) }
\end{equation*}
where is $G\left( u.v,\Delta \right) =G\left( u,v\right) .\Delta $.

Define the sequence $y=\left( y_{k}\right) ;$ which will be frequently used
as the $G\left( u.v,\Delta \right) $-transform of a sequence $x=\left(
x_{k}\right) $ i.e. 
\begin{equation}
y_{k}=\sum\limits_{i=0}^{k}u_{k}v_{i}\Delta
x_{i}=\sum\limits_{i=0}^{k}u_{k}\nabla v_{i}x_{i},\left( \nabla
v=v_{i}-v_{i+1}\right) \ \ \left( k\in N\right) .  \label{3.1}
\end{equation}%
Since the proof may also be obtained in the similar way for the other
spaces, to avoid the repetition of the similar statements, we give the proof
only for one of those spaces.

\begin{theorem}
The sequence spaces $\lambda \left( u,v,\Delta \right) $ for $\lambda \in
\left\{ \ell _{\infty },c,c_{0}\right\} $ are the complete normed linear
spaces with respect to the norm defined by 
\begin{equation}
\left\| x\right\| _{\lambda \left( u,v,\Delta \right) }=\sup_{k}\left|
\sum\limits_{i=0}^{k}u_{k}v_{i}\Delta x_{i}\right| =\left\| y\right\|
_{\lambda }  \label{3.2}
\end{equation}

\begin{proof}
The linearity of $\lambda \left( u,v,\Delta \right) $ for $\lambda \in
\left\{ \ell _{\infty },c,c_{0}\right\} $ with respect to the coordinatewise
addition and scaler multiplication follows \ the following inequalities
which are satisfied for $x$,$t\in \lambda \left( u,v,\Delta \right) $ for $%
\lambda \in \left\{ \ell _{\infty },c,c_{0}\right\} $\ and $\alpha $,$\beta
\in R$\ 
\begin{equation}
\underset{k\in N}{\sup }\left\vert \sum\limits_{i=0}^{k}u_{k}v_{i}\Delta
\left( \alpha x_{i}+\beta t_{i}\right) \right\vert \leq \left\vert \alpha
\right\vert \underset{k\in N}{\sup }\left\vert
\sum\limits_{i=0}^{k}u_{k}v_{i}\Delta x_{i}\right\vert +\left\vert \beta
\right\vert \underset{k\in N}{\sup }\left\vert
\sum\limits_{i=0}^{k}u_{k}v_{i}\Delta t_{i}\right\vert .  \label{3.3}
\end{equation}%
After this step, we must show that the spaces $\lambda \left( u,v,\Delta
\right) $ for $\lambda \in \left\{ \ell _{\infty },c,c_{0}\right\} $ hold
the norm conditions and the completeness with respect to given norm. It is
easy to show that $\left( 3.2\right) $ holds the norm condition for the
spaces $\lambda \left( u,v,\Delta \right) $ for $\lambda \in \left\{ \ell
_{\infty },c,c_{0}\right\} $. To prove the completeness of the space $\ell
_{\infty }\left( u,v,\Delta \right) $, let us take any Cauchy sequence $%
\left( x^{n}\right) $ in the space $\ell _{\infty }\left( u,v,\Delta \right) 
$. Then for a given $\varepsilon >0$, there exists a positive integer $%
N_{0}\left( \varepsilon \right) $ such that $\left\Vert
x^{n}-x^{r}\right\Vert _{\lambda \left( u,v,\Delta \right) }<\varepsilon $\
for all$\ n$, $r>N_{0}\left( \varepsilon \right) $. Hence fixed $i\in N$, 
\begin{equation*}
\left\vert G\left( u.v,\Delta \right) \left( x_{i}^{n}-x_{i}^{r}\right)
\right\vert <\varepsilon
\end{equation*}%
for all $n$,$r\geq N_{0}\left( \varepsilon \right) $. Therefore the sequence 
$\left( \left( G\Delta \right) x^{n}\right) $ is a Cauchy sequence of real
numbers for every $n\in N$. Since $R$ is complete, it converges,that is; 
\begin{equation*}
\left( \left( G\left( u.v,\Delta \right) \right) x^{r}\right) _{i\in
N}\rightarrow \left( \left( G\left( u.v,\Delta \right) \right) x\right)
_{i\in N}
\end{equation*}%
as $r\rightarrow \infty $. So we have 
\begin{equation*}
\left\vert G\left( u.v,\Delta \right) \left( x_{i}^{n}-x_{i}\right)
\right\vert <\varepsilon .
\end{equation*}%
for every $n\geq N_{0}\left( \varepsilon \right) $ and as $r\rightarrow
\infty $. This implies that $\left\Vert x^{n}-x\right\Vert _{\lambda \left(
u,v,\Delta \right) }<\varepsilon $ for every $n\geq N_{0}\left( \varepsilon
\right) $. Now we must show that $x\in \ell _{\infty }\left( u,v,\Delta
\right) $. We have 
\begin{equation*}
\underset{k}{\sup }\left\vert \left( G\left( u.v,\Delta \right) x\right)
_{k}\right\vert \leq \left\Vert x^{n}\right\Vert _{\lambda \left( u,v,\Delta
\right) }+\left\Vert x^{n}-x\right\Vert _{\lambda \left( u,v,\Delta \right)
}=O\left( 1\right)
\end{equation*}%
This implies that $x=\left( x_{i}\right) \in \ell _{\infty }\left(
u,v,\Delta \right) $. Therefore $\ell _{\infty }\left( u,v,\Delta \right) $
is a Banach space. It can be shown that $c\left( u,v,\Delta \right) $ and $%
c_{0}\left( u,v,\Delta \right) $ are closed subspaces of $\ell _{\infty
}\left( u,v,\Delta \right) $ which leads us to the consequence that the
spaces $c\left( u,v,\Delta \right) $ and $c_{0}\left( u,v,\Delta \right) $
are also the Banach spaces with the norm $(3.2)$.

Furthermore, since $\ell _{\infty }\left( u,v,\Delta \right) $ is a Banach
space with continuous coordinates, i.e ;

$\left\| G\left( u.v,\Delta \right) \left( x^{k}-x\right) \right\| _{\lambda
\left( u,v,\Delta \right) }\rightarrow 0$ implies $\left| G\left( u.v,\Delta
\right) \left( x_{i}^{k}-x_{i}\right) \right| \rightarrow 0$ for all $i\in N$%
, it is a $BK$-space.
\end{proof}
\end{theorem}

We define a Schauder basis of a normed space. If a normed sequence spaces $%
\lambda $ contains a sequence $\left( b_{n}\right) $ such that , for every $%
x\in \lambda $, there is unique sequence of scalars $\left( \alpha
_{n}\right) $ for which 
\begin{equation*}
g\left( x-\sum\limits_{k=0}^{n}\alpha _{n}b_{k}\right) \rightarrow 0\text{
as }n\rightarrow \infty \text{.}
\end{equation*}%
Then $\left( b_{n}\right) $ is called a Schauder basis for $\lambda $. The
series $\sum \alpha _{k}b_{k}$ that has the sum $x$ is called the expansion
of $x$ in $\left( b_{n}\right) $, and we write $x=\sum \alpha _{k}b_{k}$,
Maddox [\cite{Maddox} , p.98]

\begin{theorem}
Let $\lambda _{k}=\left( G\left( u.v,\Delta \right) x\right) _{k}$ for all\ $%
k\in N$. Define the sequence $b^{\left( k\right) }=\left\{ b_{n}^{\left(
k\right) }\right\} _{n\in N}$ of the elements of the space $c_{0}\left(
u,v,\Delta \right) $ by 
\begin{equation*}
b_{n}^{\left( k\right) }=\left\{ 
\begin{array}{cc}
\frac{1}{u_{n}v_{k}}-\frac{1}{u_{n}v_{k+1}} & \left( 0<k<n\right) \text{,}
\\ 
\frac{1}{u_{n}v_{n}} & \left( k=n\right) \text{,} \\ 
0 & \left( k>n\right) \text{.}%
\end{array}
\right.
\end{equation*}
for every fixed $k\in N$. Then the following assertions are true:

$i)$ The sequence $\left\{ b^{\left( k\right) }\right\} _{k\in N}$ is basis
for the space $c_{0}\left( u,v,\Delta \right) $,and any $x\in c_{0}\left(
u,v,\Delta \right) $ has a unique representation of the form 
\begin{equation*}
x=\sum\limits_{k}\lambda _{k}b^{\left( k\right) }\text{ }
\end{equation*}
$ii)$ The set $\left\{ e\text{,}b^{\left( k\right) }\right\} $ is a basis
for the space $c\left( u,v,\Delta \right) $,and any $x\in c\left( u,v,\Delta
\right) $ has a unique representation of form 
\begin{equation*}
x=le+\sum\limits_{k}\left( \lambda _{k}-l\right) b^{\left( k\right) }\text{,}
\end{equation*}
where $\lambda _{k}=\left( G\left( u.v,\Delta \right) x\right) _{k}$ $\ k\in
N$ and $l=\underset{k\rightarrow \infty }{\lim }\left( G\left( u.v,\Delta
\right) x\right) _{k}$.
\end{theorem}

\begin{theorem}
The sequence spaces $\lambda \left( u,v,\Delta \right) $ for $\lambda \in
\left\{ \ell _{\infty },c,c_{0}\right\} $ are linearly isomorphic to the
spaces $\lambda \in \left\{ \ell _{\infty },c,c_{0}\right\} $ respectively,
i.e.,

$\ell _{\infty }\left( u,v,\Delta \right) \cong \ell _{\infty }$, $c\left(
u,v,\Delta \right) \cong c$ and $c_{0}\left( u,v,\Delta \right) \cong c_{0}$.

\begin{proof}
To prove the fact $c_{0}\left( u,v,\Delta \right) \cong c_{0}$, we should
show the existence of a linear bijection between the spaces $c_{0}\left(
u,v,\Delta \right) $ and $c_{0}$. Consider the transformation $T$ defined
with the notation $(3.1)$, from $c_{0}\left( u,v,\Delta \right) $ to $c_{0}$
by $x\rightarrow y=Tx$. The linearity of $T$ from $(3.3)$ is clear.

Further, it is trivial that $x=0$ whenever $Tx=0$ and hence $T$ is injective.

Let $y\in c_{0}$ and define the sequence $x=\left\{ x_{k}\right\} $ by 
\begin{equation}
x_{k}=\sum\limits_{i=0}^{k-1}\frac{1}{u_{k}}\left( \frac{1}{v_{i}}-\frac{1}{%
v_{i+1}}\right) y_{i}+\frac{1}{u_{k}v_{k}}y_{k}\text{ \ \ \ }\left( k\in
N\right) .  \label{3.4}
\end{equation}
Then 
\begin{equation*}
\underset{k\rightarrow \infty }{\lim }\left( G\left( u.v,\Delta \right)
x\right) _{k}=\underset{k\rightarrow \infty }{\lim }\sum%
\limits_{i=0}^{k}u_{k}v_{k}\Delta x_{i}=\underset{k\rightarrow \infty }{\lim 
}y_{k}=0\text{ }
\end{equation*}
Thus we have that $x\in c_{0}\left( u,v,\Delta \right) $. Consequently, $T$
is surjective and is norm preserving. Hence, $T$ is a linear bijection which
therefore says us that the spaces $c_{0}\left( u,v,\Delta \right) $ and $%
c_{0}$ are linearly isomorphic. In the same way, it can be shown that $%
c\left( u,v,\Delta \right) $ and $\ell _{\infty }\left( u,v,\Delta \right) $
are linearly isomorphic to $c$ and $\ell _{\infty }$, respectively, and so
we omit the detail.
\end{proof}
\end{theorem}

\begin{theorem}
The sequence space $c_{0}\left( u,v,\Delta \right) $ has $AD$ property
whenever $u\in c_{0}\left( u,v,\Delta \right) $.
\end{theorem}

\begin{proof}
Suppose that $g\in \left[ c_{0}\left( u,v,\Delta \right) \right] ^{\prime }.$
Then there exists a functional $f$ over the space $c_{0}$ such that $f\left(
x\right) =g\left( G\left( u.v,\Delta \right) x\right) $ for some $f\in
c_{0}^{^{\prime }}=\ell _{1}.$ Since $c_{0}$ has $AK-property$ and $%
c_{0}^{^{\prime }}\cong \ell _{1}$%
\begin{equation*}
f\left( x\right) =\sum\limits_{j=1}^{\infty
}a_{j}\sum\limits_{i=1}^{j}u_{j}\nabla v_{i}x_{i}
\end{equation*}
for some $a=\left( a_{j}\right) \in \ell _{1}.$ Since $u\in c_{0}$ by
hypothesis the inclusion $\varphi \subset c_{0}\left( u,v,\Delta \right) $
holds. For any $f\in c_{0}^{^{\prime }}$ and $e^{\left( k\right) }\in
\varphi \subset c_{0}\left( u,v,\Delta \right) ,$ 
\begin{equation*}
f\left( e^{\left( k\right) }\right) =\sum\limits_{j=1}^{\infty }a_{j}\left\{
G\left( u.v,\Delta \right) e^{\left( k\right) }\right\} _{j}=\left\{
G^{^{\prime }}\left( u.v,\Delta \right) a\right\} _{k}
\end{equation*}
where $G^{^{\prime }}\left( u.v,\Delta \right) $ is the transpose of the
matrix $G\left( u.v,\Delta \right) .$ Hence, from Hanh-Banach Theorem, $%
\varphi $ is dens in $c_{0}\left( u,v,\Delta \right) $ if and only if $%
G^{^{\prime }}\left( u.v,\Delta \right) a=\theta $ for $a\in \ell _{1}$
implies $a=\theta .$ Since null space of the operator $G^{^{\prime }}\left(
u.v,\Delta \right) $ on $w$ is $\left\{ \theta \right\} ,$ $c_{0}\left(
u,v,\Delta \right) $ has $AD$ property.
\end{proof}

For the sequence spaces $\lambda $ and $\mu $, define the set $S\left(
\lambda ,\mu \right) $ by 
\begin{equation}
S\left( \lambda ,\mu \right) =\left\{ z=\left( z_{k}\right) \in w:xz=\left(
x_{k}z_{k}\right) \in \mu \text{ for all }x\in \lambda \right\}  \label{3.5}
\end{equation}
With notation of $(3.5)$, the $\alpha -$, $\beta -$, and $\gamma -$ duals of
a sequence space $\lambda $, which are respectively denoted by $\lambda
^{\alpha }$, $\lambda ^{\beta }$, and $\lambda ^{\gamma }$ are defined in 
\cite{Garlin} by 
\begin{equation*}
\lambda ^{\alpha }=S\left( \lambda ,l_{1}\right) \text{, }\lambda ^{\beta
}=S\left( \lambda ,cs\right) \text{ and }\lambda ^{\gamma }=S\left( \lambda
,bs\right) \text{.}
\end{equation*}
We now need the following Lemmas due to Stieglitz and Tietz \cite{stiti} for
the proofs of theorems $5-6-7$.

\begin{lemma}
$A\in \left( c_{0}:l_{1}\right) $ if and only if

$\underset{K\in F}{\sup }\sum\limits_{n}\left\vert \sum\limits_{k\in
K}a_{nk}\right\vert <\infty ,$
\end{lemma}

\begin{lemma}
$A\in \left( c_{0}:c\right) $ if and only if\ 

$\ \underset{n}{\sup }\sum\limits_{k}\left\vert a_{nk}\right\vert <\infty ,$

$\underset{n\rightarrow \infty }{\lim }a_{nk}-\alpha _{k}=0.$
\end{lemma}

\begin{lemma}
$A\in \left( c_{0}:\ell _{\infty }\right) $ if and only if

$\underset{n}{\sup }\sum\limits_{k}\left| a_{nk}\right| <\infty .$
\end{lemma}

\begin{theorem}
Let $u$, $v\in U$, $a=\left( a_{k}\right) \in w$\ and the matrix $B=\left(
b_{nk}\right) $ by 
\begin{equation*}
b_{nk}=\left\{ 
\begin{array}{cc}
\left( \frac{1}{u_{n}v_{k}}-\frac{1}{u_{n}v_{k+1}}\right) a_{k} & \left(
0\leq k\leq n\right) \text{,} \\ 
\frac{1}{u_{n}v_{n}}a_{n} & \left( k=n\right) \text{,} \\ 
0 & \left( k>n\right) \text{.}%
\end{array}
\right.
\end{equation*}
for all $k,$ $n\in N.$ Then the $\alpha -$ dual of the \ space $\lambda
\left( u,v;\Delta \right) $\ is the set 
\begin{equation*}
b_{\Delta }=\left\{ a=\left( a_{k}\right) \in w:\underset{K\in F}{\sup }%
\sum\limits_{n}\left| \sum\limits_{k\in K}\sum\limits_{i=0}^{k-1}\frac{1}{%
u_{k}}\left( \frac{1}{v_{i}}-\frac{1}{v_{i+1}}\right) a_{k}+\frac{1}{%
u_{k}v_{k}}a_{k}\right| <\infty \right\} \text{.}
\end{equation*}
\end{theorem}

\begin{proof}
Let $a=\left( a_{k}\right) \in w$ and consider the matrix $G^{-1}\left(
u.v,\Delta \right) =\Delta ^{-1}G^{-1}\left( u,v\right) $ and sequence $%
a=\left( a_{k}\right) $. Bearing in mind the relation $(3.1)$, we
immediately derive that 
\begin{eqnarray}
a_{k}x_{k} &=&\sum\limits_{i=0}^{k-1}\frac{1}{u_{k}}\left( \frac{1}{v_{i}}-%
\frac{1}{v_{i+1}}\right) y_{i}a_{k}+\frac{1}{u_{k}v_{k}}y_{k}a_{k}
\label{3.6} \\[1pt]
&=&\sum\limits_{i=0}^{k-1}\frac{1}{u_{k}}\left( \frac{1}{v_{i}}-\frac{1}{%
v_{i+1}}\right) a_{i}y_{k}+\frac{1}{u_{k}v_{k}}y_{k}a_{k}=\left( By\right)
_{k}  \notag
\end{eqnarray}
($i,k\in N$). We therefore observe by $(3.6)$ that $ax=\left(
a_{n}x_{n}\right) \in \ell _{1}$ whenever $x\in c_{0}\left( u,v,\Delta
\right) $, $c\left( u,v,\Delta \right) $ and $\ell _{\infty }\left(
u,v,\Delta \right) $ if and only if $By\in \ell _{1}$ whenever $y\in \left\{
\ell _{\infty },c,c_{0}\right\} $. Then , we derive by Lemma $1$ that 
\begin{equation*}
\underset{K\in F}{\sup }\sum\limits_{n}\left| \sum\limits_{k\in
K}\sum\limits_{i=0}^{k-1}\frac{1}{u_{k}}\left( \frac{1}{v_{i}}-\frac{1}{%
v_{i+1}}\right) a_{k}+\frac{1}{u_{k}v_{k}}a_{k}\right| <\infty
\end{equation*}
which yields the consequence that $\left[ c_{0}\left( u,v,\Delta \right) %
\right] ^{\alpha }=\left[ c\left( u,v,\Delta \right) \right] ^{\alpha }=%
\left[ \ell _{\infty }\left( u,v,\Delta \right) \right] ^{\alpha }=b_{\Delta
}$.
\end{proof}

\begin{theorem}
Let $u$, $v\in U$, $a=\left( a_{k}\right) \in w$ and the matrix $C=\left(
c_{nk}\right) $ by 
\begin{equation*}
c_{nk}=\left\{ 
\begin{array}{cc}
\left( \frac{1}{u_{n}v_{k}}-\frac{1}{u_{n}v_{k+1}}\right) a_{k} & \left(
0\leq k<n\right) \text{,} \\ 
\frac{1}{u_{n}v_{n}}a_{n} & \left( k=n\right) \text{,} \\ 
0 & \left( k>n\right) \text{.}%
\end{array}%
\right.
\end{equation*}%
and define the sets $c_{1}$,$c_{2}$,$c_{3}$,$c_{4}$ by 
\begin{eqnarray*}
c_{1} &=&\left\{ a=\left( a_{k}\right) \in w:\underset{n}{\sup }%
\sum\limits_{n}\left\vert c_{nk}\right\vert <\infty \right\} ;\text{\ }%
c_{2}=\left\{ a=\left( a_{k}\right) \in w:\underset{n\rightarrow \infty }{%
\lim }c_{nk}\text{\ exists for each }k\in N\right\} ; \\
c_{3} &=&\left\{ a=\left( a_{k}\right) \in w:\underset{n\rightarrow \infty }{%
\lim }\sum\limits_{k}\left\vert c_{nk}\right\vert =\sum\limits_{k}\left\vert 
\underset{n\rightarrow \infty }{\lim }c_{nk}\right\vert \right\} ;\text{ }%
c_{4}=\left\{ a=\left( a_{k}\right) \in w:\underset{n\rightarrow \infty }{%
\lim }\sum\limits_{k}c_{nk}\text{ exists }\right\} \text{.}
\end{eqnarray*}%
Then, $\left[ c_{0}\left( u,v,\Delta \right) \right] ^{\beta }$,$\left[
c\left( u,v,\Delta \right) \right] ^{\beta }$and $\left[ \ell _{\infty
}\left( u,v,\Delta \right) \right] ^{\beta }$ is $c_{1}\cap c_{2}$, $%
c_{1}\cap c_{2}\cap c_{4}$ and $c_{2}\cap c_{3}$ respectively.
\end{theorem}

\begin{proof}
We only give the proof the space $c_{0}\left( u,v,\Delta \right) $. Since
the proof may be obtained by the same way for the spaces $c\left( u,v,\Delta
\right) $ and $\ell _{\infty }\left( u,v,\Delta \right) $ Consider the
equation 
\begin{eqnarray}
\sum\limits_{i=0}^{n}a_{k}x_{k} &=&\sum\limits_{i=0}^{n}\left[
\sum\limits_{i=0}^{k-1}\frac{1}{u_{k}}\left( \frac{1}{v_{i}}-\frac{1}{v_{i+1}%
}\right) y_{i}+\frac{1}{u_{k}v_{k}}y_{k}\right] a_{k}  \label{3.7} \\
&=&\sum\limits_{i=0}^{n}\left[ \sum\limits_{i=0}^{n}\frac{1}{u_{k}}\left( 
\frac{1}{v_{i}}-\frac{1}{v_{i+1}}\right) a_{i}+\frac{1}{u_{k}v_{k}}a_{k}%
\right] y_{k}  \notag \\
&=&\left( Cy\right) _{n}\text{.}  \notag
\end{eqnarray}%
Thus , we deduce from lemma $2$ and $\left( 3.7\right) $ that $ax=\left(
a_{n}x_{n}\right) \in cs$\ whenever $x\in c_{0}\left( u,v,\Delta \right) $
if and only if $Cy\in c$ whenever $y\in c_{0}.$ Therefore we derive by lemma 
$2$. which shows that $\left\{ c_{0}\left( u,v,\Delta \right) \right\}
^{\beta }=c_{1}\cap c_{2}$.
\end{proof}

\begin{theorem}
The $\gamma -$dual of the $c_{0}\left( u,v,\Delta \right) $, $c\left(
u,v,\Delta \right) $ and $\ell _{\infty }\left( u,v,\Delta \right) $ is the
set $c_{1}$.

\begin{proof}
This may be obtained in the similar way used in the proof of \ Theorem 6
with Lemma 3 instead of Lemma 2. So, we omit the detail .
\end{proof}
\end{theorem}

\section{{\protect\Large Matrix Transformations on the Space }$c\left(
u,v,\Delta \right) $}

In this section, we directly prove the theorems which characterize the
classes $\left( c\left( u,v,\Delta \right) :\ell _{\infty }\right) $ and $%
A\in \left( c\left( u,v,\Delta \right) :c\right) $.

\begin{theorem}
$A\in \left( c\left( u,v,\Delta \right) :\ell _{\infty }\right) $ if and
only if 
\begin{equation}
\underset{n}{\sup }\sum\limits_{i=0}^{n}\left| \sum\limits_{i=0}^{k-1}\frac{1%
}{u_{k}}\left( \frac{1}{v_{i}}-\frac{1}{v_{i+1}}\right) a_{nk}+\frac{1}{%
u_{k}v_{k}}a_{nk}\right| <\infty  \label{4.1}
\end{equation}
\begin{equation}
\underset{n\rightarrow \infty }{\lim }\left[ \sum\limits_{i=0}^{k-1}\frac{1}{%
u_{k}}\left( \frac{1}{v_{i}}-\frac{1}{v_{i+1}}\right) a_{nk}+\frac{1}{%
u_{k}v_{k}}a_{nk}\right]  \label{4.2}
\end{equation}
exists for all $k$,$n\in N$. 
\begin{equation}
\underset{n\in N}{\sup }\sum\limits_{k=0}^{n}\left| \sum\limits_{i=0}^{k-1}%
\frac{1}{u_{k}}\left( \frac{1}{v_{i}}-\frac{1}{v_{i+1}}\right) a_{nk}+\frac{1%
}{u_{k}v_{k}}a_{nk}\right| <\infty ,(n\in N).  \label{4.3}
\end{equation}
\begin{equation}
\underset{n\rightarrow \infty }{\lim }\sum\limits_{k=0}^{n}\left[
\sum\limits_{i=0}^{k-1}\frac{1}{u_{k}}\left( \frac{1}{v_{i}}-\frac{1}{v_{i+1}%
}\right) a_{nk}+\frac{1}{u_{k}v_{k}}a_{nk}\right]  \label{4.4}
\end{equation}
\ exists for all $n\in N$.

\begin{proof}
Let $A\in \left( c\left( u,v,\Delta \right) :\ell _{\infty }\right) $. Then $%
Ax$ exists and is in $\ell _{\infty }$ for all $x\in c\left( u,v,\Delta
\right) $. Thus, since $\left\{ a_{nk}\right\} _{k\in N}\in \left\{ c\left(
u,v,\Delta \right) \right\} ^{\beta }$ for all $n\in N$, the necessities of
the conditions \ (4.2)-(4.4) are satisfy. Because of $Ax$ exists and is in $%
\ell _{\infty }$ for every $x\in c\left( u,v,\Delta \right) $.

Let us consider the equality 
\begin{equation*}
\sum\limits_{k=0}^{n}a_{nk}x_{k}=\sum\limits_{k=0}^{n}\left[
\sum\limits_{i=0}^{k-1}\frac{1}{u_{k}}\left( \frac{1}{v_{i}}-\frac{1}{v_{i+1}%
}\right) +\frac{1}{u_{k}v_{k}}\right] a_{nk}y_{k\text{ }}\text{;}\left( n\in
N\right)
\end{equation*}%
which yields us under our assumptions as $n\rightarrow \infty $\ that 
\begin{equation}
\sum\limits_{k=0}^{\infty }a_{nk}x_{k}=\sum\limits_{k=0}^{\infty }\left[
\sum\limits_{i=0}^{k-1}\frac{1}{u_{k}}\left( \frac{1}{v_{i}}-\frac{1}{v_{i+1}%
}\right) a_{nk}+\frac{1}{u_{k}v_{k}}a_{nk}\right] y_{k}\text{, \ \ \ }(n\in
N).  \label{4.5}
\end{equation}%
Therefore we get by (4.5) that 
\begin{eqnarray*}
\left\Vert Ax\right\Vert _{\infty } &\leq &\underset{n}{\sup }%
\sum\limits_{k}\left\vert \sum\limits_{i=0}^{k-1}\frac{1}{u_{k}}\left( \frac{%
1}{v_{i}}-\frac{1}{v_{i+1}}\right) a_{nk}+\frac{1}{u_{k}v_{k}}%
a_{nk}\right\vert \left\vert y_{k}\right\vert \\
&\leq &\left\Vert y\right\Vert _{\infty }\underset{n}{\sup }%
\sum\limits_{k}\left\vert \sum\limits_{i=0}^{k-1}\frac{1}{u_{k}}\left( \frac{%
1}{v_{i}}-\frac{1}{v_{i+1}}\right) a_{nk}+\frac{1}{u_{k}v_{k}}%
a_{nk}\right\vert <\infty .
\end{eqnarray*}%
This means that $Ax\in \ell _{\infty }$ whenever $x\in c\left( u,v,\Delta
\right) $ and this step completes the proof .
\end{proof}
\end{theorem}

We wish to give a lemma concerning the characterization of the class $A\in
\left( c_{0}:c\right) $ which is needed in proving the next theorem and due
to Stieglitz and Tietz\cite{stiti}.

\begin{lemma}
$A\in \left( c_{0}:c\right) $ if and only if the conditions of Lemma 2 with $%
\alpha _{k}=0$ for all $k$, and 
\begin{equation*}
\underset{n\rightarrow \infty }{\lim }\sum\limits_{k}a_{nk}=0\text{.}
\end{equation*}

\begin{theorem}
$A\in \left( c\left( u,v,\Delta \right) :c\right) $ if and only if
(4.1)-(4.4) hold, and 
\begin{equation}
\underset{n\rightarrow \infty }{\lim }\underset{k}{\sum }\sum%
\limits_{i=0}^{k-1}\frac{1}{u_{k}}\left( \frac{1}{v_{i}}-\frac{1}{v_{i+1}}%
\right) a_{nk}+\frac{1}{u_{k}v_{k}}a_{nk}=\alpha ,  \label{4.6}
\end{equation}%
\begin{equation}
\underset{n\rightarrow \infty }{\lim }\left( \sum\limits_{i=0}^{k-1}\frac{1}{%
u_{k}}\left( \frac{1}{v_{i}}-\frac{1}{v_{i+1}}\right) a_{nk}+\frac{1}{%
u_{k}v_{k}}a_{nk}\right) =\alpha _{k},(k\in N).  \label{4.7}
\end{equation}
\end{theorem}

\begin{proof}
Let $A\in \left( c\left( u,v,\Delta \right) :c\right) $. Since $c\subset
\ell _{\infty }$, the necessities of (4.1)-(4.4) are immediately obtain
Theorem 4.1. $Ax$ exists and is in $c$ for all $x\in c\left( u,v,\Delta
\right) $ by the hypothesis, so in necessities of the conditions (4.6) and
(4.7) are easily obtained with the sequences $x=\left( 1,2,3,...\right) $
and $x=b^{\left( i\right) }$, respectively; where $b^{\left( i\right) }$ is
defined by (1.2). Conversely, suppose that the conditions (4.1)-(4.4), (4.6)
and (4.7) hold and take any $x\in c\left( u,v,\Delta \right) $. Then $%
\left\{ a_{ni}\right\} _{i\in N}\in \left\{ c\left( u,v,\Delta \right)
\right\} ^{\beta }$ for each $n\in N$ which implies that $Ax$ exists. One
can derive by (4.1) and (4.7) that 
\begin{equation*}
\overset{s}{\underset{i=0}{\sum }}\left\vert \alpha _{i}\right\vert \leq 
\underset{n}{\sup }\sum\limits_{k=0}^{n}\left\vert \sum\limits_{i=0}^{k-1}%
\frac{1}{u_{k}}\left( \frac{1}{v_{i}}-\frac{1}{v_{i+1}}\right) a_{nk}+\frac{1%
}{u_{k}v_{k}}a_{nk}\right\vert <\infty
\end{equation*}%
holds every $s\in N$. This yields that $\left( \alpha _{i}\right) \in \ell
_{1}$ and hence the series $\sum \alpha _{i}y_{i}$ absolutely converges. Let
us consider the following equality obtained from (3.5) with $a_{ni}-\alpha
_{i}$ instead of $a_{ni}$%
\begin{equation}
\sum\limits_{k}\left( a_{nk}-\alpha _{k}\right)
x_{k}=\sum\limits_{k=0}^{\infty }\left[ \sum\limits_{i=0}^{k-1}\frac{1}{u_{k}%
}\left( \frac{1}{v_{i}}-\frac{1}{v_{i+1}}\right) a_{nk}+\frac{1}{u_{k}v_{k}}%
a_{nk}\right] \left( a_{nk}-\alpha _{i}\right) y_{k}.  \label{4.8}
\end{equation}%
Therefore we derive from lemma 2 with (4.8) that 
\begin{equation}
\underset{n\infty }{\lim }\sum\limits_{k}\left( a_{nk}-\alpha _{k}\right)
x_{k}=0\text{.}  \label{4.9}
\end{equation}%
Thus, we deduce by combining (4.9) with the fact $\left( \alpha
_{i}y_{i}\right) \in \ell _{1}$ that $Ax\in c$ and this step completes the
proof.
\end{proof}
\end{lemma}


\begin{thebibliography}{99}
\bibitem{ZUM1987} Z. U. Ahmad and Mursaleen, K\"{o}the-Toeplitz duals of
some new sequence spaces and their matrix maps, \textit{Publ. Inst. Math.
(Beograd)} 42$\left( 56\right) $ $\left( 1987\right) $ $57-61$.

\bibitem{BAHP2006} B. Altay and H. Polat, On some new Euler difference
sequence spaces, \textit{Southeast Asian Bull. Math.} $30$ $\left(
2006\right) $ 209-220.

\bibitem{ALBA2007} B. Altay and F. Ba\c{s}ar, The fine spectrum and the
matrix domain of the difference operator $\Delta $ on the sequence space $%
\ell _{p},\left( 0<p<\infty \right) ,$ \textit{Commun. Math. Anal. }2$\left(
2\right) \left( 2007\right) ,$ 1-11.

\bibitem{BAFB2006} B. Altay and F. Ba\c{s}ar, Some paranormed sequence
spaces of nonabsolute type derived by weighted, J. Math. Anal. Appl.
319(2006) 494-508.

\bibitem{C.Et} R. \c{C}olak and M. Et, On some generalized difference
sequence spaces and related matrix transformations,\textit{\ Hokkaido Math J,%
} ,26(3),$\left( 1997\right) $ 483-492.

\bibitem{vatan polat} V. Karakaya and H. Polat, Some New Paranormed Sequence
Spaces defined by Euler and Difference Operators, \textit{Acta Sci.
Math(Szeged)}, 7z(2010), 209--222.

\bibitem{HKIZ1981} H. K\i zmaz, On certain sequence space, \textit{Canad.
Math. Bull.} $24\left( 2\right) $ $\left( 1981\right) $ $169-176$.

\bibitem{Maddox} I. J. Maddox, \textit{Elements of Functional Analysis,} The
University Press, 2$^{nd}$ ed., Cambridge, $1988$.

\bibitem{HPFB2007} H. Polat and F. Ba\c{s}ar, Some Euler spaces of
difference sequences of order $m$, \textit{Acta Math. Sci. Ser. B England
Ed. }$\ 27B\left( 2\right) ~\left( 2007\right) ~254-266$.

\bibitem{W.R} W. H. Ruckle, \textit{Sequence spaces}, Pitman Publishing,
Toronto $1981$.

\bibitem{MS2004} E. Malkowsky and E. Sava\c{s}, Matrix transformations
between sequence spaces of generalized weighted mean, Appl. Math.Comput. 147
(2004) 333-345.

\bibitem{Garlin} D. J. H. Garling, The $\alpha -,$ $\beta -,$ $\gamma -$%
duality sequence spaces, Proc. Camb. Phil. Soc., 63(1967), 963-981.

\bibitem{stiti} M. Stieglitz and H. Tietz, Matrix transformationen von
Folgenraumen Eine Ergebnis\"{u}bersict, Math. Z. 154 (1977) 1-16.

\bibitem{ETBASAR} [10] M. Et and M. Ba\c{s}ar\i r, On some genaralized
difference sequence spaces, Period. Math.

Hung. 35 (3) (1997), 169- 175.

\bibitem{Wilansky} A. Wilansky, Summability Through Functional Analysis,
North-Holland Mathematics Studies, vol. 85, North-Holland, Amsterdam, 1984.
\end{thebibliography}
\end{document}